\documentclass[notitlepage,11pt,reqno]{amsart}
\usepackage[foot]{amsaddr}
\usepackage{amssymb,nicefrac,bm,upgreek,mathtools,verbatim,enumerate}
\usepackage[mathscr]{eucal}
\usepackage{dsfont}
\usepackage[normalem]{ulem}
\usepackage{amsopn}

\usepackage[margin=1in]{geometry}
\allowdisplaybreaks

\newcommand{\stkout}[1]{\ifmmode\text{\sout{\ensuremath{#1}}}\else\sout{#1}\fi}

\newtheorem{lemma}{Lemma}[section]
\newtheorem{theorem}{Theorem}[section]

\newtheorem{corollary}{Corollary}[section]

\theoremstyle{definition}

\newtheorem{notation}{Notation}[section]

\theoremstyle{remark}
\newtheorem{remark}{Remark}[section]

\usepackage{amsthm}
\newtheoremstyle{sltheorem}
{}                
{}                
{\slshape}        
{}                
{\bfseries}       
{.}               
{ }               
{}                
\theoremstyle{sltheorem}

\numberwithin{theorem}{section}
\numberwithin{equation}{section}

\usepackage[pdftex,breaklinks]{hyperref}
\hypersetup{
  colorlinks=true,
  citecolor=mblue,
  linkcolor=mblue,
  frenchlinks=false,
  urlcolor=mblue,
  pdfborder={0 0 0},
  naturalnames=false,
  hypertexnames=false,
  breaklinks}
\usepackage[capitalize,nameinlink]{cleveref}
\usepackage[abbrev,msc-links,nobysame,citation-order]{amsrefs} 

\crefname{section}{Section}{Sections}
\crefname{subsection}{Section}{Sections}
\crefname{condition}{Condition}{Conditions}
\crefname{hypothesis}{Hypothesis}{Conditions}
\crefname{assumption}{Assumption}{Assumptions}
\crefname{lemma}{Lemma}{Lemmas}
\crefname{fact}{Fact}{Facts}

\Crefname{figure}{Figure}{Figures}

\crefformat{equation}{\textup{#2(#1)#3}}
\crefrangeformat{equation}{\textup{#3(#1)#4--#5(#2)#6}}
\crefmultiformat{equation}{\textup{#2(#1)#3}}{ and \textup{#2(#1)#3}}
{, \textup{#2(#1)#3}}{, and \textup{#2(#1)#3}}
\crefrangemultiformat{equation}{\textup{#3(#1)#4--#5(#2)#6}}%
{ and \textup{#3(#1)#4--#5(#2)#6}}{, \textup{#3(#1)#4--#5(#2)#6}}%
{, and \textup{#3(#1)#4--#5(#2)#6}}

\Crefformat{equation}{#2Equation~\textup{(#1)}#3}
\Crefrangeformat{equation}{Equations~\textup{#3(#1)#4--#5(#2)#6}}
\Crefmultiformat{equation}{Equations~\textup{#2(#1)#3}}{ and \textup{#2(#1)#3}}
{, \textup{#2(#1)#3}}{, and \textup{#2(#1)#3}}
\Crefrangemultiformat{equation}{Equations~\textup{#3(#1)#4--#5(#2)#6}}%
{ and \textup{#3(#1)#4--#5(#2)#6}}{, \textup{#3(#1)#4--#5(#2)#6}}%
{, and \textup{#3(#1)#4--#5(#2)#6}}

\crefdefaultlabelformat{#2\textup{#1}#3}
\newcommand{\vertiii}[1]{{\left\vert\kern-0.25ex\left\vert\kern-0.25ex\left\vert #1 
    \right\vert\kern-0.25ex\right\vert\kern-0.25ex\right\vert}}
\newcommand{\lamstr}{\lambda^{\mspace{-2mu}*}}
\newcommand{\rc}{{F}}    

\newcommand{\Uadm}{\mathscr{A}}

\newcommand{\Ag}{{\mathcal{A}}}  
\newcommand{\fB}{{\mathfrak{B}}}  
\newcommand{\Cc}{{C}}   
\newcommand{\sF}{{\mathfrak{F}}}   
\newcommand{\eom}{{\mathcal{M}}} 
\newcommand{\Lg}{{\mathcal{L}}}    
\newcommand{\Lp}{{L}}            

\newcommand{\cP}{{\mathcal{P}}}  

\newcommand{\RR}{\mathds{R}}
\newcommand{\NN}{\mathds{N}}

\newcommand{\Rd}{{\mathds{R}^{d}}}
\DeclareMathOperator{\Exp}{\mathbb{E}}
\DeclareMathOperator{\Prob}{\mathbb{P}}
\newcommand{\D}{\mathrm{d}}

\newcommand{\Ind}{\mathds{1}}   

\newcommand{\Sob}{{\mathscr W}}    
\newcommand{\Sobl}{{\mathscr W}_{\text{loc}}} 

\newcommand{\df}{\coloneqq}

\DeclareMathOperator*{\dist}{dist}
\DeclareMathOperator*{\Argmin}{Arg\,min}

\newcommand{\abs}[1]{\lvert#1\rvert}
\newcommand{\norm}[1]{\lVert#1\rVert}
\newcommand{\babs}[1]{\bigl\lvert#1\bigr\rvert}

\usepackage{color}
\definecolor{dmagenta}{rgb}{.4,.1,.5}
\definecolor{dblue}{rgb}{.0,.0,.5}
\definecolor{mblue}{rgb}{.0,.0,.6}
\definecolor{ddblue}{rgb}{.0,.0,.4}
\definecolor{dred}{rgb}{.6,.0,.0}
\definecolor{dgreen}{rgb}{.0,.5,.0}
\definecolor{Eeom}{rgb}{.0,.0,.5}

\newcommand{\ttl}{\Large
On uniqueness of solutions to viscous HJB equations\\[3pt] with a subquadratic 
nonlinearity in the gradient}

\begin{document}
\title[A note on viscous Hamilton-Jacobi equations]{\ttl}

\author[Ari Arapostathis]{Ari Arapostathis$^*$}
\address{$^*$ Department of Electrical and Computer Engineering,
The University of Texas at Austin, 2501 Speedway, EER~7.824,
Austin, TX~~78712, USA}
\email{ari@utexas.edu}
\author[Anup Biswas]{Anup Biswas$^\dag$}
\address{$^\dag$ Department of Mathematics,
Indian Institute of Science Education and Research,
Dr. Homi Bhabha Road, Pune 411008, India}
\email{anup@iiserpune.ac.in}
\author[Luis Caffarelli]{Luis Caffarelli$^\ddag$}
\address{$^\ddag$Department of Mathematics,
The University of Texas at Austin, 2515 Speedway, RLM 10.150,
Austin, TX 78712}
\email{caffarel@math.utexas.edu}

\begin{abstract}
Uniqueness of positive solutions to viscous Hamilton--Jacobi--Bellman (HJB) equations
of the form
$-\Delta u(x) + \frac{1}{\gamma} \abs{D{u}(x)}^\gamma  =
f(x) - \lambda$, 
with $f$ a coercive function and $\lambda$ a constant,
in the subquadratic case, that is, $\gamma\in(1,2)$,
appears to be an open problem.
Barles and Meireles [\emph{Comm. Partial Differential Equations} \textbf{41} (2016)]
show uniqueness in the case that $f(x) \approx \abs{x}^\beta$
and $\abs{D f(x)}\lessapprox \abs{x}^{(\beta-1)_+}$ for some
$\beta>0$, essentially matching earlier results of Ichihara,
who considered more general Hamiltonians but with better regularity for $f$.
Without enforcing this assumption, to our knowledge,
there are no results on uniqueness in the literature.
In this short article, we show
that the equation has a unique positive solution
for any locally Lipschitz continuous, coercive $f$ which satisfies
$\babs{D f(x)} \le \kappa\bigl(1 +\abs{f(x)}^{2-\nicefrac{1}{\gamma}}\bigr)$
for some positive constant $\kappa$.
Since $2-\frac{1}{\gamma}>1$, this assumption imposes very mild restrictions on
the growth of the potential $f$.
We also show that this solution fully characterizes
optimality for the associated ergodic problem.
Our method involves the study of an infinite dimensional linear program for
elliptic equations for measures,
and is very different from earlier approaches.
It also applies to the larger class of Hamiltonians studied by Ichihara,
and we show that it is well suited to provide verification of optimality
results for the associated ergodic control problems, even in a pathwise sense,
and without resorting to the parabolic problem.
\end{abstract}



\maketitle


\section{Introduction}

We consider the viscous Hamilton--Jacobi--Bellman (HJB) equation
\begin{equation}\label{EP}
-\Delta u(x) + \frac{1}{\gamma} \abs{D{u}(x)}^\gamma  \,=\,
f(x) - \lambda\,,\tag{\sf{EP}}
\end{equation}
for $(u,\lambda)\in\Cc^2(\Rd)\times\RR$, with $\gamma>1$.
Here, $f\in\Sobl^{1,\infty}(\Rd)$ and is coercive.
By coercive, sometimes
also called \emph{inf-compact}, we refer to a function $f$
whose sublevel sets
$\{x\in\Rd\,\colon f(x)\le r\}$ are compact (or empty) for every $r\in\RR$.
As shown in \cite{Barles-16}, \cref{EP} has a classical solution $u$ for any
$\lambda\le\lamstr$, where
\begin{equation}\label{E-lamstr}
\lamstr \,\df\, \sup\,\bigl\{\lambda\in\Rd\,\colon
\text{\cref{EP}\ has a subsolution}\bigr\}\,.
\end{equation}

This equation which has a long history in the literature,
has been studied in \cites{Ichihara-12,Cirant-14} for somewhat
more general Hamiltonians, and was recently revisited by Barles in \cite{Barles-16}.
What is of interest here, is to characterize the solutions of
\cref{EP} which are bounded from below, that is, without loss of generality,
the positive solutions.
Naturally,
when we refer to this equation having a unique positive solution,
we mean that the solution is unique up to an additive constant.
In the superquadratic case ($\gamma\ge2$), \cite[Theorem~2.6]{Barles-16} shows
that \cref{EP} has a unique positive solution for any coercive $f$,
and in addition, for this solution, $\lambda=\lamstr$.
In the subquadratic case ($\gamma\in(1,2)$), \cite{Barles-16} adopts assumption
(H2) in \cite{Ichihara-12},
which states that $f$ satisfies a bound of the form
\begin{equation}\label{IH2}
c^{-1}\abs{x}^{\beta} - c \,\le\,
f(x) \,\le\, c\bigl(1 +\abs{x}^\beta\bigr)\,,\quad\text{and\ \ }
\babs{D f(x)} \,\le\, c^{-1} \Bigl (1 +\abs{x}^{(\beta-1)_+}\Bigr)
\end{equation}
for some positive constants $\beta$ and $c$ for all $x\in\Rd$.
Without enforcing this assumption, to our knowledge,
there are no results on uniqueness in the literature, and therefore also no
verification of optimality results.
Note that \cite{Cirant-14} introduces an additional stable drift to study the
subquadratic case.

There is substantial literature on viscous HJB equations, other than
\cite{Barles-16,Ichihara-12,Cirant-14} mentioned above.
It is not our intent to review this literature, since it does not
address the problem studied in this paper, but we should at least mention
\cite{Barles-01,Ben-Fre,BenNag-91,Barles-10,Fujita-06a,Ichihara-13a,BenFre-92,Lasry-89}.

We adopt the following assumption for $\gamma\in(1,2)$.

\medskip
\begin{itemize}
\item[\hypertarget{A1}{\sf{(A1)}}]
{\slshape The function $f$ is locally Lipschitz continuous and coercive, and there
exists a constant $\kappa_0$ such that}
\begin{equation*}
\babs{D f(x)} \,\le\, \kappa_0\Bigl(1 +\abs{f(x)}^{2-\frac{1}{\gamma}}\Bigr)
\qquad\forall\,x\in \Rd\,.
\end{equation*}
\end{itemize}
\medskip

We show that, under \hyperlink{A1}{\sf{(A1)}},
there exists a unique positive solution $u$ to \cref{EP}.
In addition, this solution
fully characterizes the ergodic control problem in the sense that
a stationary Markov control is optimal if and only if it agrees a.e.\ on $\Rd$
with the function $\xi_u$ in \cref{EL2.3A} (see \cref{T4.1}).
The method we follow covers the more general Hamiltonians studied in
\cites{Ichihara-12,Cirant-14}, and also improves the
existing results for the superquadratic case.
This is discussed in \cref{S3}.

\subsection{Brief summary of the method}

Consider the operator $\Ag\colon\Cc^2(\Rd)\to\Cc(\Rd\times\Rd)$ defined by
\begin{equation}\label{EAg}
\Ag g(x,\xi) \,\df\, -\Delta g(x) + \xi\cdot D g(x)\,,\quad\text{for\ }
(x,\xi)\in\Rd\times\Rd\,.
\end{equation}
Let $\cP(\Rd\times\Rd)$ denote the space of probability measures on the Borel
$\sigma$-algebra of $\Rd\times\Rd$, denoted as $\fB(\Rd\times\Rd)$,
endowed with the Prokhorov topology.
We say that $\mu\in\cP(\Rd\times\Rd)$ is \emph{infinitesimally invariant} for the
operator $\Ag$ if $\int \Ag g\,\D\mu =0$ for all $g\in\Cc_c^2(\Rd)$, the
latter denoting the functions in $\Cc^2(\Rd)$ with compact support, and denote
the set of these probability measures by $\eom$.
Let
\begin{equation}\label{EF}
\rc(x,\xi) \,\df\, f(x) + \frac{1}{\gamma^*} \abs{\xi}^{\gamma^*} \,,\qquad
\text{with\ \ }
\gamma^* \,\df\, \frac{\gamma}{\gamma-1}\,.
\end{equation}
For $\mu\in\cP(\Rd\times\Rd)$ we use the simple notation
\begin{equation*}
\mu(F)\,\df\, \int_{\Rd\times\Rd} \rc(x,\xi)\,\mu(\D x,\D\xi)\,,
\end{equation*}
and define
\begin{equation}\label{EMF}
\eom_F\,\df\, \bigl\{\mu\in\eom\,\colon \mu(F)<\infty\bigr\}\,.
\end{equation}
In other words, $\eom_F$ is the subset of $\eom$ consisting of those probability
measures under which $F$ is integrable.
It is simple to show that $\eom_F$ is always nonempty.
Thus, since $F$ is coercive, the set
\begin{equation*}
\eom_{F,r}\,\df\, \bigl\{\mu\in\eom\,\colon \mu(F)\le r\bigr\}
\end{equation*}
is compact  for all $r>0$ sufficiently large.
Clearly, it is also convex.

Consider the minimization problem
\begin{equation}\label{LP}
\overline\lambda \,\df\, \inf_{\mu\in\eom}\;\mu(F)\,.\tag{\sf{LP}}
\end{equation}

The lower semicontinuity of $\mu\mapsto \mu(F)$ then implies that
the infimum of \cref{LP} is attained in $\eom$.
We let $\eom_F^*$ denote the set of points in $\eom$ which attain
this infimum.

Our approach to the proof of uniqueness of positive solutions of \cref{EP}
is as follows:  First, we show that if $(u,\lambda)\in\Cc^2(\Rd)\times\RR$
is any pair solving \cref{EP}, with $u$ a positive function, then $\lambda=\overline\lambda$
and some measure $\mu\in\eom$
taking the form $\mu(\D{x},\D\xi) = \nu(\D{x}) \delta_{D{u}(x)}(\D\xi)$,
with $\nu\in\cP(\Rd)$ and $\delta_{D{u}(x)}$ denoting the Dirac mass
at $D{u}(x)$, attains the infimum in \cref{LP}, that
is, it belongs to $\eom_F^*$.
Next, we show that $\eom_F^*$ is a singleton, thus establishing the uniqueness
of a positive solution to \cref{EP}.

\subsection{Notation}
The standard Euclidean norm in $\RR^{d}$ is denoted by $\abs{\,\cdot\,}$,
and $\NN$ stands for the set of natural numbers.
The closure, the boundary and the complement
of a set $A\subset\RR^{d}$ are denoted
by $\Bar{A}$, $\partial{A}$ and $A^{c}$, respectively.
The open ball of radius $r$ in $\RR^{d}$, centered at $x\in\Rd$,
is denoted by $B_{r}(x)$, and $B_r$ is the ball centered at $0$.
We use $a_{\pm}\df \max (\pm a,0)$ for $a\in\RR$.

For a Borel space $Y$, $\cP(Y)$ denotes the set of probability measures
on its Borel $\sigma$-algebra, and $\delta_y$ denotes the Dirac mass at $y\in Y$.
For $\mu\in\cP(Y)$ and a measurable function $g\colon Y\to\RR$ which is
integrable under $\mu$, we often use the simplifying notation
$\mu(g)\df \int_{Y} f\,\D\mu$.

\section{Main results}\label{S2}

Throughout this section we assume $\gamma\in(1,2)$, unless otherwise explicitly mentioned.
Also, without loss of generality  we assume that $f\ge1$, and we scale
a solution of \cref{EP}, which is bounded from below,
by an additive constant so that $\inf_\Rd\, u=1$.

We start with the very useful gradient estimate in \cite[Theorem~B.1]{Ichihara-12},
stated under weaker regularity in \cite[Theorem~A.2]{Barles-16} for \cref{EP}.
It appears that the Bernstein approach for
 this estimate originates in \cite[Theorem~A.1]{Lasry-89}.
We need to use this estimate on small balls, so we scale it as follows.

\begin{corollary}\label{C2.1}
There exists a constant $C$ such that for any solution $u$
of \cref{EP} we have
\begin{equation}\label{EC2.1A}
\sup_{z\in B_r(x)}\,\abs{D{u}(z)} \,\le\, C\,\biggl(r^{-\frac{1}{\gamma-1}} +
\sup_{z\in B_{2r}(x)}\,\bigl(f(z)-\lambda\bigr)_+^{\nicefrac{1}{\gamma}}
+ \sup_{z\in B_{2r}(x)}\,\babs{D f(z)}^{\nicefrac{1}{(2\gamma-1)}}\biggr)
\quad \forall\,x\in\Rd\,,
\end{equation}
and for all $r>0$.
In particular, under \hyperlink{A1}{\sf{(A1)}},
with perhaps a different constant $C_0$, we have
\begin{equation}\label{EC2.1B}
\sup_{z\in B_r(x)}\,\abs{D{u}(z)} \,\le\,
C_0\,\biggl(r^{-\frac{1}{\gamma-1}} + \sup_{z\in B_{2r}(x)}\,
\bigl(f(y)-\lambda\bigr)_+^{\nicefrac{1}{\gamma}}\biggr)
\quad \forall\,x\in\Rd\,,\ \ \forall\,r>0\,.
\end{equation}
\end{corollary}

\begin{proof}
Fix any $x\in\Rd$.
For $r>0$, let $u_r(y)\df r^{\frac{2-\gamma}{\gamma-1}} u(x+ry)$
and $f_r(y) \df r^{\gamma^*}\bigl(f(x+ry) - \lambda\bigr)$.
The function $u_r$ satisfies
\begin{equation}\label{PC2.1A}
-\Delta u_r(y) + \frac{1}{\gamma} \babs{D{u}_r(y)}^\gamma  \,=\, 
f_r(y)\,.
\end{equation}
By \cite[Theorem~B.1]{Ichihara-12}, there exists a constant $C$ such that
any solution $u_r$ of \cref{PC2.1A} satisfies
\begin{equation*}
\begin{aligned}
\sup_{y\in B_1(x)}\,\abs{D{u_r}(y)} &\,\le\, C\,\biggl(1 +
\sup_{y\in B_{2}(x)}\,\bigl(f_r(y)\bigr)_+^{\nicefrac{1}{\gamma}}
+ \sup_{y\in B_{2}(x)}\,\babs{D f_r(y)}^{\nicefrac{1}{(2\gamma-1)}}\biggr)
\quad \forall\,x\in\Rd\,.
\end{aligned}
\end{equation*}
from which \cref{EC2.1A} follows.
%
\end{proof}

We continue by proving a useful lower bound for positive solutions of
\cref{EP}.
Define
\begin{equation*}
\Gamma_x \,\df\, \bigl[f(x)\bigr]^{-\frac{1}{\gamma^*}}\,,\quad x\in\Rd\,.
\end{equation*}

\begin{lemma}\label{L2.1}
Assume \hyperlink{A1}{\sf{(A1)}}.
Then, for every positive solution
$u$ of \cref{EP}, the following hold.
\begin{enumerate}
\item[\textup{(}a\textup{)}]
There exist positive constants $r$ and $\kappa$ such that
\begin{equation}\label{EL2.1A}
\inf_{y\in B_{r}}\, u\bigl( x+ \Gamma_x\, y) \,\ge\,
\kappa\, \bigl[f(x)\bigr]^{\frac{\gamma^*-2}{\gamma^*}} \quad\forall\,x\in\Rd\,.
\end{equation}
\item[\textup{(}b\textup{)}]
There exists a positive constant $M_0$ such that
\begin{equation}\label{EL2.1B}
 \frac{\abs{D{u}(x)}^2}{u(x) }\,\le\, M_0\,f(x)\qquad\forall\,x\in\Rd\,.
\end{equation}
\end{enumerate}
\end{lemma}

\begin{proof}
Note that by \hyperlink{A1}{\sf{(A1)}} there exists some $R>0$ such that
\begin{equation}\label{PL2.1A}
\abs{D f(x_n)} \,\le\, 2\kappa_0 \bigl[f(x_n)\bigr]^{\frac{\gamma^*+1}{\gamma^*}}
\qquad\forall\,x\in B_{R}^c\,.
\end{equation}
Choose $r$ positive and small enough such that
$r\le \frac{\gamma^*}{8\kappa_0}$.
We claim that the assertion in part (a) holds for this $r$.
To prove this, we use contradiction.
Suppose that
\begin{equation}\label{PL2.1B}
\inf_{y\in B_r}\, u(x_n+ \Gamma_{x_n} y)\bigl[f(x_n)\bigr]^{\frac{2-\gamma^*}{\gamma^*}}
\,\xrightarrow[n\to\infty]{}\,0
\end{equation}
along some sequence $\{x_n\}_{n\in\NN}\subset\Rd$, such that $\abs{x_n}\to\infty$
as $n\to\infty$.
We write \cref{EP} as
\begin{equation*}
-\Delta u(x) +  b(x)\cdot D{u}(x) \,=\,  f(x) - \lambda
\end{equation*}
with $b(x) \df \frac{1}{\gamma}\abs{D{u}}^{\gamma-2}D{u}(x)$.
Simplifying the notation, let
$\Gamma_n\equiv \Gamma_{x_n} = \bigl[f(x_n)\bigr]^{-\frac{1}{\gamma^*}}$, and
define the sequence of scaled functions
\begin{equation*}
u_n(y) \,\df\, \frac{u\bigl(x_n + \Gamma_n y\bigl)}
{\bigl[f(x_n)\bigr]^{\frac{\gamma^*-2}{\gamma^*}}}\,,\quad
b_n(y) \,\df\, \Gamma_n\, b\bigl(x_n + \Gamma_n y\bigl)\,,\quad\text{and\ \ }
f_n(y)\,\df\,\frac{f\bigl(x_n + \Gamma_n y\bigl) - \lambda}{f(x_n)}\,,
\end{equation*}
for $y\in \Rd$ and $n\in\NN$.
Then we obtain from \cref{EP} that
\begin{equation}\label{PL2.1C}
-\Delta u_n(y) +  b_n(y)\cdot D{u}_n(y) \,=\, f_n(y)\,.
\end{equation}

Integrating \cref{PL2.1A}, we obtain
\begin{equation}\label{PL2.1D}
\bigl[f(x_n+\Gamma_n y)\bigr]^{-\frac{1}{\gamma^*}} -
\bigl[f(x_n)\bigr]^{-\frac{1}{\gamma^*}}\,\ge\,
-\frac{2\kappa_0}{\gamma^*}\,\abs{y}\,\Gamma_n\,.
\end{equation}
Computing also the lower bound inherited from \cref{PL2.1A}
and combining it with \cref{PL2.1D}, we obtain
\begin{equation}\label{PL2.1E}
\bigl(\tfrac{2}{3}\bigr)^{\gamma^*}\,
f(x_n)  \,\le\, f(x_n+\Gamma_n y) \,\le\,
2^{\gamma^*} f(x_n) \quad\forall\, y\in B_{4r}\,,\quad\forall\,x_n\in
B_{R}^c\,.
\end{equation}
This shows that $f_n$ and $f_n^{-1}$ are bounded in $B_{4r}$ uniformly in $n\in\NN$.
To establish a bound for $b_n$ on $B_{2r}$,
it is enough to show that, for some constant $C$, we have
\begin{equation}\label{PL2.1F}
\abs{D{u}(x_n+\Gamma_n y)} \,\le\, C (1 + \Gamma_n^{1-\gamma^*})
\,=\, C \bigl(1 + \bigl[f(x_n)\bigr]^{\nicefrac{1}{\gamma}}\bigr)
\quad \forall\,y\in B_{2r}\,.
\end{equation}
By \cref{EC2.1B} we have
\begin{equation}\label{PL2.1G}
\begin{aligned}
\sup_{y\in B_{2r}}\,\abs{D{u} (x_n +\Gamma_n y)} &\,\le\,
C_0\,\biggl(\bigl(2\Gamma_n r\bigr)^{-\frac{1}{\gamma-1}} +
\sup_{y\in B_{4r}}\,\bigl(f(x_n+\Gamma_n y)-\lambda\bigr)_+^{\nicefrac{1}{\gamma}}\biggr)\\
&\,=\,
C_0\,\biggl((2r)^{-\frac{1}{\gamma-1}}
\bigl[f(x_n)\bigr]^{\frac{1}{\gamma}} +
\sup_{y\in B_{4r}}\,\bigl(f(x_n+\Gamma_n y)-\lambda\bigr)_+^{\nicefrac{1}{\gamma}}\biggr)
\,.
\end{aligned}
\end{equation}
Thus \cref{PL2.1F} follows by \cref{PL2.1E,PL2.1G}.

Therefore, since, as we have shown, $f_n$, $f_n^{-1}$ and $b_n$ are bounded in
$B_{2r}$ uniformly in $n\in\NN$, then,
by using for example \cite[Lemma~3.6]{AA-Harnack},
we see that equation \cref{PL2.1C} contradicts the
hypothesis in \cref{PL2.1B} that $\inf_{y\in B_r}\, u_n(y)\to0$.
This completes the proof of part (a).

Moving to part (b), let $r$ be as chosen in the proof of part (a).
We have shown above that 
\begin{equation}\label{PL2.1H}
\sup_{x\in\Rd}\,\sup_{y\in B_{4r}}\,
\frac{f(x + \Gamma_x y) - \lambda}{f(x)}\,<\infty\,.
\end{equation}
On the other hand, using the estimate \cref{PL2.1G}
on $B_{\nicefrac{r}{2}}$, with $r$ and $R$ as in part (a), we have
\begin{equation*}
\begin{aligned}
\frac{\abs{D{u}(x)}^2}{u(x) f(x)} &\,\le\,
\frac{C_0}{u(x) f(x)}
\biggl((\nicefrac{r}{2})^{-\frac{1}{\gamma-1}}\bigl[f(x)\bigr]^{\frac{1}{\gamma}}
+ \sup_{y\in B_{r}}\,\bigl(f(x+\Gamma_x y)-\lambda\bigr)_+^{\nicefrac{1}{\gamma}}\biggr)^2
\\
&\,\le\, 2 C_0 \Biggl((\nicefrac{r}{2})^{-\frac{1}{\gamma-1}}
\frac{\bigl[f(x)\bigr]^{\frac{\gamma^*-2}{\gamma^*}}}{u(x)}
+ \sup_{y\in B_{r}}\,
\frac{\bigl[f(x+\Gamma_x y)\bigr]^{\frac{\gamma^*-2}{\gamma^*}}}{u(x)} \times
\frac{f(x+\Gamma_x y)}{f(x)} \Biggr)
\quad \forall\,x\in\Rd\,.
\end{aligned}
\end{equation*}
Therefore, \cref{EL2.1B} follows by \cref{EL2.1A,PL2.1H}.
This completes the proof.
\end{proof}

\begin{remark}
The estimate in \cref{L2.1} is not suitable for the superquadratic case.
A different scaling can be used when $\gamma\ge2$.
First, we replace \hyperlink{A1}{\sf{(A1)}} by
\begin{equation}\label{ER2.1A}
\babs{D f(x)} \,\le\, \kappa_0\Bigl(1 +\abs{f(x)}^{\frac{4\gamma-3}{3\gamma-2}}\Bigr)
\qquad\forall\,x\in \Rd\,.
\end{equation}
Then, under \cref{ER2.1A}, we obtain
\begin{equation}\label{ER2.1B}
\inf_{y\in B_{r}}\, u\bigl( x+ \Gamma_x\, y) \,\ge\,
\kappa\, \bigl[f(x)\bigr]^{\frac{\gamma}{3\gamma-2}} \quad\forall\,x\in\Rd\,.
\end{equation}
for some positive constants $r$ and $\kappa$.
To prove this, we use $\Gamma_x = \bigl[f(x)\bigr]^{\frac{1-\gamma}{3\gamma-2}}$,
and follow the proof of \cref{L2.1}.
\end{remark}

To continue, we need the following notation.

\begin{notation}
For $r>0$, we let $\chi^{}_r$ be a concave $\Cc^2(\RR)$ function such that
$\chi^{}_r(t)= t$ for $t\le r$, and $\chi'_r(t) = 0$ for $t\ge 3r$.
Then $\chi'_r$ and $-\chi''_r$ are nonnegative, and the latter is supported
 on $[r,3r]$.
In addition, we select $\chi^{}_r$ so that
\begin{equation}\label{Echi}
\abs{\chi''_r(t)} \,\le\, \frac{1}{t}\qquad\forall t>0\,.
\end{equation}
This is always possible.
For example, we can specify $\chi''_r$ as
\begin{equation*}
\chi''_r(t) = \begin{cases}
\frac{t-r}{r^2} & \text{if\ }  r\le t \le \frac{3r}{2}\,,\\[3pt]
\frac{1}{2r} & \text{if\ }  \frac{3r}{2}\le t \le \frac{5r}{2}\,,\\[3pt]
\frac{3}{r} - \frac{t}{r^2} &\text{if\ }  \frac{5r}{2}\le t \le 3r\,.
\end{cases}
\end{equation*}
\end{notation}

Recall the definitions in \cref{EF,EMF,LP}.

\begin{lemma}\label{L2.2}
Assume \hyperlink{A1}{\sf{(A1)}}.
For any positive solution $u\in\Cc^2(\Rd)$
of \cref{EP} with eigenvalue $\lambda$  and $\mu\in\eom_F$, we have
\begin{equation}\label{EL2.2A}
\mu(F) -\lambda\,=\,
\int_{\Rd\times\Rd} \Bigl(\tfrac{1}{\gamma^*} \abs{\xi}^{\gamma^*}
-\xi\cdot D{u}(x)
+ \tfrac{1}{\gamma} \babs{ D{u}(x)}^\gamma\Bigr)\,\mu(\D x,\D\xi)\,\ge\,0\,,
\end{equation}
In particular, $\lambda\le\overline\lambda$.
\end{lemma}

\begin{proof}
Since $u$ is coercive by \cref{L2.1}\,(a),
it follows that $\chi^{}_r(u)-r-1$ is compactly supported.
Thus we have
\begin{equation*}
\int \bigl(\Delta \chi^{}_r(u) - \xi\cdot D\chi^{}_r(u)\bigr)\,\D\mu
\,=\,0\qquad \forall\,\mu\in\eom\,,
\end{equation*}
by the definition of $\eom$.
On the other hand, we have
\begin{align*}
\Delta \chi^{}_r(u) - \xi\cdot D\chi^{}_r(u) &\,=\,
\chi''_r(u) \abs{ D{u}}^2
+ \chi'_r(u)\bigl(\Delta u -\xi\cdot D{u}\bigr)\\
&\,=\, \chi''_r(u) \abs{ D{u}}^2
+\chi'_r(u)\Bigl(\lambda + \tfrac{1}{\gamma} \abs{ D{u}}^\gamma - f
-\xi\cdot D{u}\Bigr)\\
&\,=\, \chi''_r(u) \abs{ D{u}}^2
+\chi'_r(u)\Bigl(\lambda - f-  \tfrac{1}{\gamma^*} \abs{\xi}^{\gamma^*}\Bigr)\\
&\mspace{200mu}+ \chi'_r(u)\Bigl(\tfrac{1}{\gamma^*} \abs{\xi}^{\gamma^*}
 -\xi\cdot D{u}
+ \tfrac{1}{\gamma} \abs{ D{u}}^\gamma\Bigr)\,.
\end{align*}
Therefore,
\begin{equation*}
\begin{aligned}
\int \chi'_r(u)\Bigl( f + \tfrac{1}{\gamma^*} \abs{\xi}^{\gamma^*}-\lambda\Bigr)\,\D\mu
&\,=\, \int \chi''_r(u) \abs{ D{u}}^2\,\D\mu\\
&\mspace{40mu}
+ \int \chi'_r(u)\Bigl(\tfrac{1}{\gamma^*} \abs{\xi}^{\gamma^*} -\xi\cdot D{u}
+ \tfrac{1}{\gamma} \abs{ D{u}}^\gamma\Bigr)\,\D\mu
\end{aligned}
\end{equation*}
for all $\mu\in\eom$.
By \cref{L2.1} we have $\frac{\abs{ D{u}}^2}{u}\le M_0 f$ for some
constant $M_0$.
Using this together with \cref{Echi}, we obtain
\begin{align*}
\int \chi''_r(u) \abs{ D{u}}^2\,\D\mu &\,\le\,
\int \Ind_{\{x\colon u(x)\ge\, r\}}  \frac{\babs{ D{u}(x)}^2}{u(x)}\,\D\mu\\
&\,\le\,
M_0 \int \Ind_{\{x\colon u(x)\ge\, r\}}\,f(x)\,\D\mu
\,\xrightarrow[r\to\infty]{}\,0\,,
\end{align*}
since $\int f\D\mu<\infty$ by the definition of $\eom_F$.
Thus letting $r\nearrow\infty$,
and applying the monotone convergence theorem, we obtain \cref{EL2.2A},
thus completing the proof.
\end{proof}

It is convenient to express the operator $\Ag$ in \cref{EAg} in terms of
a family of operators $\{\Lg_\xi\}_{\xi\in\Rd}$ defined by
\begin{equation}\label{ELg}
\Lg_\xi g(x) \,\df\, -\Delta g(x) + \xi\cdot D g(x)\,,\quad g\in\Cc^2(\Rd)
\end{equation}
It is clear from the Legendre--Fenchel transform
\begin{equation*}
\max_{\xi\in\Rd}\, \bigl(\xi\cdot p - \tfrac{1}{\gamma^*} \abs{\xi}^{\gamma^*}\Bigr)
\,=\, \tfrac{1}{\gamma} \abs{p}^{\gamma}\,,
\end{equation*}
that, for any solution $u$ of \cref{EP}, we have
\begin{equation*}
\max_{\xi\in\Rd}\, \Bigl[\Lg_\xi u(x) - \tfrac{1}{\gamma^*} \abs{\xi}^{\gamma^*}\Bigr]
\,=\, -\Delta u(x) + \frac{1}{\gamma} \babs{D u(x)}^\gamma
\,=\, f(x)-\lambda\,,
\end{equation*}
and that the maximum is realized at $\xi(x)=\abs{D u(x)}^{\gamma-2} D u(x)$.

The next lemma applies to any $\gamma>1$.

\begin{lemma}\label{L2.3}
Let $\gamma>1$.
Let $u$ be a coercive solution of \cref{EP} with eigenvalue $\lambda$.
Define
\begin{equation}\label{EL2.3A}
\xi_u(x)\,\df\, \abs{ D{u}(x)}^{\gamma-2}\, D{u}(x)\,.
\end{equation}
Then, there exists a Borel
probability measure $\nu_u$ on $\fB(\Rd)$, such that
\begin{equation}\label{EL2.3B}
\mu_u(\D{x},\D\xi)\,\df\, \nu_u(\D{x})\delta_{\xi_u(x)}(\D\xi) \,\in\,\eom\,,
\end{equation}
and
\begin{equation}\label{EL2.3C}
\int_{\Rd\times\Rd} \rc(x,\xi)\,\mu_u(\D x,\D\xi) \,\le\, \lambda\,.
\end{equation}
In particular, $\overline\lambda\le\lambda$.
\end{lemma}

\begin{proof}
Using the definition in \cref{ELg}, we write
\begin{equation}\label{PL2.3A}
\Lg_{\xi_u(x)}u(x)\,=\,
-\Delta u(x) + \abs{D u(x)}^{\gamma-2} D u(x) \,=\, F(x,\xi) - \lambda\,.
\end{equation}
Since $u$ is coercive, we can apply \cite[Theorem~1.2]{Bogachev-01d} to assert
the existence of a unique $\nu_u\in\cP(\Rd)$ which satisfies
\begin{equation}\label{PL2.3B}
\int_\Rd \Lg_{\xi_u(x)} f(x)\,\nu_u(\D{x})\,=\,0\qquad\forall\,f\in\Cc_c^2(\Rd)\,.
\end{equation}
Thus \cref{EL2.3B} follows from \cref{PL2.3B} and the definition of $\mu_u$,
whereas \cref{EL2.3C} follows by integrating \cref{PL2.3A} with respect to
$\nu_u$ using a cut-off function as in the proof of \cref{L2.2}, which
shows that $\int_\Rd \Lg_{\xi_u(x)}g(x)\,\D\mu_u\le0$ for any positive
function $g\in\Cc^2(\Rd)$.
\end{proof}

\begin{remark}
\Cref{L2.3} can also be established by a simple probabilistic argument.
Viewing \cref{EL2.3B} as a Foster--Lyapunov equation, it is well known
that the coercivity of $u$ implies that the diffusion with extended
generator $\Lg_{\xi_u(x)}$ is positive recurrent.
The measure $\nu_u$ can then be specified as the unique invariant probability
measure of this diffusion.
\end{remark}

We now discuss some properties of the set $\eom$ of infinitesimally
invariant measures which are needed for
the proof of \cref{T2.1} below.
It is clear that every $\mu\in\eom$ can be disintegrated into a
probability measure $\nu(\D{x})\in\cP(\Rd)$ and a Borel measurable probability kernel
$\eta(x,\D\xi)$ on $\Rd\times\fB(\Rd)$.
We denote this disintegration by $\mu=\nu\circledast \eta$.
For $\nu\circledast \eta\in\eom$, define
$\Bar\eta(x) \df \int_\Rd \xi\, \eta(x,\D\xi)$.
Then $\Bar\eta\colon\Rd\to\Rd$ is a Borel measurable map.
It is straightforward
to verify that $\nu\circledast \delta_{\Bar\eta}$ is in $\eom$.
Since, by convexity, we have
\begin{equation*}
\int_{\Rd\times\Rd}\abs{\xi}^\gamma \mu(\D{x},\D\xi)
\,\ge\, \int_\Rd\abs{\Bar\eta(x)}^\gamma\,\nu(\D{x})\,,
\end{equation*}
it is clear that the infimum in \cref{LP} is attained at
some $\mu\in\eom$  whose disintegration results in a kernel
$\eta(x,\,\cdot\,)$ which is Dirac for each $x\in\Rd$.
Then, $\eta$ can be represented as a Borel measurable map $v\colon\Rd\to\Rd$,
and vice-versa.
We denote the class of such measures as $\overline\eom$, and
abusing the notation we represent them as $\mu=\nu\circledast v$.
Consider such a $\mu=\nu\circledast v$ in $\overline\eom\cap\eom_F$.
It follows by \cref{EL2.2A} that $\int_\Rd \abs{v(x)}^{\gamma^*}\,\eta(\D{x})<\infty$.
Since $\gamma^*> 2$, this implies that $\int_\Rd \abs{v(x)}^2\,\eta(\D{x})<\infty$,
and thus $\nu$ has density $\varrho\in\Lp^{\nicefrac{d}{(d-1)}}(\Rd)$
with respect to the Lebesgue measure by  \cite[Theorem~1.1]{Bogachev-96}.

We continue with our main theorem.
Recall the definition of $\lamstr$ in \cref{E-lamstr},
and that $\eom_F^*$ denotes the subset of $\eom$ consisting of points
that attain the infimum in \cref{LP}.

\begin{theorem}\label{T2.1}
Assume \hyperlink{A1}{\sf{(A1)}}.  The following hold.
\begin{enumerate}
\item[\textup{(}a\textup{)}]
For any positive solution $u$ of \cref{EP} with eigenvalue $\lambda$ we have
\begin{equation*}
\mu_u(F) \,=\, \lambda \,=\, \overline\lambda \,=\,\lamstr\,,
\end{equation*}
with $\mu_u$ as in \cref{L2.3}.
\item[\textup{(}b\textup{)}]
The set $\eom_F^*$ is a singleton.
\item[\textup{(}c\textup{)}]
There exists a unique positive solution of \cref{EP}.
\end{enumerate}
\end{theorem}

\begin{proof}
By \cref{L2.1} every positive solution of \cref{EP} is coercive.
Therefore, the first two equalities in part~(a) follow by \cref{L2.2,L2.3}.
By \cite[Theorem~2.6]{Barles-16}, there exists a solution with eigenvalue
$\lamstr$ which is bounded from below.
This of course implies $\overline\lambda=\lamstr$ thus completing the proof
of part~(a).

Let $\nu_u$ and $\xi_u$ be as in \cref{L2.3}.
Let $\mu=\nu\circledast v$ be any element of $\eom_F^*\cap\overline\eom$.
By the discussion in the paragraph preceding the theorem, $\nu$ has a density
$\varrho\in\Lp^{\nicefrac{d}{(d-1)}}(\Rd)$ with respect to the Lebesgue measure.
Let $\varrho_u$ denote the density of $\nu_u$, which, as well known, is strictly positive.
Let
\begin{equation*}
\zeta\,\df\,\frac{\rho}{\rho+\rho_u}\,,\quad
\zeta_u\df\frac{\rho_u}{\rho+\rho_u}\,,\quad
\Bar{v} \df \zeta v + \zeta_u \xi_u\,,\quad\text{and\ }
\Bar\nu\df\frac{1}{2}(\nu+\nu_u)\,.
\end{equation*}
It is straightforward to verify that
$\Bar\nu\circledast\Bar{v}\in\eom_\circ$.

By optimality, we have
\begin{equation}\label{PT2.1A}
\begin{aligned}
0 &\;\le\; \int_{\Rd} \abs{\Bar{v}}^{\gamma^*}\,\D\Bar\nu
-\frac{1}{2}\int_{\Rd} \abs{v}^{\gamma^*}\,\D\nu -\frac{1}{2}\int_{\Rd}
\abs{\xi_u}^{\gamma^*}\,\D\nu_u\\[3pt]
&\;=\; \int_\Rd \abs{\zeta\, v + \zeta_u\, \xi_u}^{\gamma^*}\,\D\Bar\nu
- \tfrac{1}{2}\,\int_\Rd \abs{v}^{\gamma^*}\, \D\nu
- \tfrac{1}{2}\,\int_\Rd \abs{\xi_u}^{\gamma^*}\, \D\nu_u \\[3pt]
&\;=\; \int_{\Rd}\Bigl(\abs{\zeta\, v + \zeta_u\, \xi_u}^{\gamma^*}
-\zeta\, \abs{v}^{\gamma^*} -\zeta_u\, \abs{\xi_u}^{\gamma^*}\Bigr)\, \D\Bar\nu
\,\le\,0
\end{aligned}
\end{equation}
by convexity.
Thus $\Bar\nu\circledast\Bar{v}\in\eom_F^*$.
Since $\rho_u$ is strictly positive, \eqref{PT2.1A} implies
that $v=\xi_u$ on the support of $\varrho$.
It is clear that if $v$ is modified outside the support of $\varrho$, then
the modified measure is also infinitesimally invariant for $\Ag$.
Therefore $\nu\circledast \xi_u\in\eom_F^*$.
The uniqueness of a probability measure satisfying \cref{PL2.3B}
then implies that $\nu=\nu_u$,
which in turn implies (since $v=\xi_u$ on the support of $\nu$)
that $v=\xi_u$ a.e.\ in $\Rd$.
This completes the proof of part (b).

Turning to part (c), existence of a positive
solution follows from \cite[Theorem~2.6]{Barles-16}.
By part (b), for any positive solutions $u$ and $w$, we have
$\xi_u = \xi_w$ a.e.\ in $\Rd$, implying that $D u=D w$ on $\Rd$.
Thus the solution is unique up to an additive constant.
This completes the proof.
\end{proof}

\section{More general Hamiltonians}\label{S3}

In this section we consider viscous equations 
taking the form
\begin{equation}\label{EP2}
-\Delta u(x) + H(x, D{u})  \,=\, f(x) - \lambda\,,
\end{equation}
with more general Hamiltonians $H$.
We adopt the following assumptions.

\medskip
\begin{itemize}
\item[\hypertarget{A2}{\sf{(A2)}}]
{\slshape The function $f$ is in $\Cc^2(\Rd)$ and is coercive.
The Hamiltonian $H$ satisfies the following.
\begin{enumerate}
\item[\textup{(i)}]
$H\in \Cc^2(\Rd\times\bigl(\Rd\setminus\{0\})\bigr)$,
and $p\mapsto H(x,p)$ is strictly convex for all $x\in\Rd$.
\item[\textup{(ii)}]
There exist constants $h_0>0$ and $\gamma>1$, such that
\begin{equation*}
h_0^{-1} \abs{p}^\gamma-h_0 \,\le\, H(x,p)\,\le\, h_0
\bigl(\abs{p}^\gamma+1\bigr)\,,
\qquad \babs{D_x H(x,p)} \,\le\, h_0^{-1} \bigl(1+\abs{p}^\gamma\bigr)\,.
\end{equation*}
\end{enumerate}
}
\end{itemize}
\medskip

The hypothesis \hyperlink{A2}{\sf{(A2)}}
is equivalent to \hyperlink{A2p}{\sf{(A2$^\prime$)}} below for the Lagrangian $L$,
which is related to $H$ via the Fenchel--Legendre transform, that is,
\begin{equation}\label{ELF2}
H(x,p) \,=\, \sup_{\xi\in\Rd}\, \bigl(\xi\cdot p - L(x,\xi)\bigr)\,.
\end{equation}

\medskip
\begin{itemize}
\item[\hypertarget{A2p}{\sf{(A2$^\prime$)}}]
{\slshape The function $f$ is in $\Cc^2(\Rd)$ and is coercive.
The Lagrangian $L$ satisfies the following.
\begin{enumerate}
\item[\textup{(i)}]
$L\in \Cc^2(\Rd\times\bigl(\Rd\setminus\{0\})\bigr)$,
and $\xi\mapsto L(x,\xi)$ is strictly convex for all $x\in\Rd$.
\item[\textup{(ii)}]
There exist constants $l_0>0$ and $\gamma^*>1$, such that
\begin{equation*}
l_0^{-1}  \abs{\xi}^{\gamma^*} - l_0 \,\le\, L(x,\xi) \,\le\,
l_0 \bigl(\abs{\xi}^{\gamma^*}+1)\,,
\qquad \babs{D_x L(x,\xi)} \,\le\, l_0^{-1} \bigl(1+\abs{\xi}^{\gamma^*}\bigr)\,.
\end{equation*}
\end{enumerate}
}
\end{itemize}
\medskip

In addition, under \hyperlink{A2}{\sf{(A2)}} or \hyperlink{A2}{\sf{(A2$^\prime$)}},
there exists positive constants
$h_1$ and $l_1$
such that
\begin{equation}\label{ELF3}
\begin{aligned}
h_1 \abs{p}^{\gamma-1} - h_1^{-1}
&\,\le\, \babs{D_p H(x,p)}\,\le\, h_1^{-1} \bigl(\abs{p}^{\gamma-1}+1\bigr)\,,\\
l_1 \abs{\xi}^{\gamma^*-1} - l_1^{-1} &\,\le\, \babs{D_\xi L(x,\xi)}\,\le\,
l_1^{-1} \bigl(\abs{\xi}^{\gamma^*-1}+1\bigr)\,,
\end{aligned}
\end{equation}
for all $(x,p,\xi)\in\RR^{3d}$, and
\begin{equation}\label{ELF4}
H(x,p) + L(x,\xi) \,\ge\, \xi\cdot p\,,
\end{equation}
with equality if and only if $\xi=D_p H(x,p)$ or $p=D_\xi L(x,\xi)$.

The model above is slightly more general than the model in \cite{Ichihara-12,Cirant-14}.
A more restrictive assumption on $H$ is used in \cite{Ichihara-12}, while
 $H$ does not depend on $x$ in \cite{Cirant-14}.
For the properties mentioned above see \cite[Theorem~3.4]{Ichihara-12}
and \cite[Proposition~4.1]{Cirant-14}.

\medskip

As mentioned earlier, \cite{Ichihara-12} imposes the assumptions in \cref{IH2} for $f$,
for both the subquadratic and superquadratic cases.
Barles in \cite{Barles-16} uses \cref{IH2}
only for the subquadratic case, while \cite{Cirant-14} does not
consider unbounded $f$ for the subquadratic case.
Analogous is the model in \cite[Section~4.6]{Ben-Fre}.

The results for this Hamiltonian are essentially the same as those in \cref{S2}.
We need the following ramification of \cite[Theorem~B.1]{Ichihara-12} analogous to
\cref{C2.1} valid for solutions of \cref{EP2}.

\begin{corollary}\label{C3.1}
Assume \hyperlink{A2}{\sf{(A2)}}.
Then, there exists a constant $C$ such that any solution $u$
of \cref{EP2} satisfies \cref{EC2.1A}.
\end{corollary}

\begin{proof}
A closer inspection of the proof of \cite[Theorem~B.1]{Ichihara-12} reveals that the
following is established.
Let $g\in\Cc^2(\Rd)$ be a coercive function.
There exists a function $C\colon (0,\infty)^2\to(0,\infty)$ such that
if $\varphi\in\Cc^3(\Rd)$ satisfies
\begin{equation}\label{EC3.1A}
\begin{aligned}
-D (\Delta \varphi) &\,\le\, c_1\bigl(1 + \abs{D \varphi}^\gamma + \abs{D g}\bigr)\,,\\
\abs{D\varphi}^\gamma &\,\le\, c_2\bigl(\abs{\Delta\varphi} + \abs{g}\bigr)\,,
\end{aligned}
\end{equation}
for a pair of positive constants $(c_1,c_2)$, then
\begin{equation}\label{EC3.1B}
\sup_{y\in B_1(x)}\,\abs{D{u}(y)} \,\le\, C(c_1,c_2)\biggl(1 +
\sup_{y\in B_2(x)}\,\bigl(g(y)\bigr)_+^{\nicefrac{1}{\gamma}}
+ \sup_{y\in B_2(x)}\,\abs{D g(y)}^{\nicefrac{1}{(2\gamma-1)}}\biggr)
\qquad \forall\,x\in\Rd\,.
\end{equation}
We use scaling.  With $u$ a solution of \cref{EP2}, we define
$u_r(y)\df r^{\frac{2-\gamma}{\gamma-1}} u(x+ry)$.
Using \hyperlink{A2}{\sf{(A2)}}\,(ii) and \cref{ELF3}, we deduce
that $u_r$ and $g\equiv r^{\gamma^*}\bigl(f(x+ry) - \lambda\bigr)$ satisfy \cref{EC3.1A}
for all $r\in(0,1]$ and for constants $c_1$ and $c_2$ which do not depend on $r$.
The result then follows by \cref{EC3.1B}.
\end{proof}

For the model in \cref{EP2}, we define
\begin{equation}\label{EMF2}
F(x,\xi) \,\df\, f(x) + L(x,\xi)\,,
\end{equation}
and $\eom_F$ and $\overline\lambda$ as in \cref{EMF}
and \cref{LP}, respectively, relative to $F$ in \cref{EMF2}.
We also let
\begin{equation}\label{Exiu}
\xi_u(x)\,\df\, D_p H(x,p)\,.
\end{equation}
Recall that $\eom_F^*$ is the set of measures in $\eom$ that attain
the infimum in \cref{LP}.

\begin{theorem}\label{T3.1}
Assume \hyperlink{A1}{\sf{(A1)}}--\hyperlink{A2}{\sf{(A2)}} and $\gamma\in(1,2)$.
Then
\begin{enumerate}
\item[\textup{(}a\textup{)}]
The conclusions of \cref{L2.1} hold.
\item[\textup{(}b\textup{)}]
For any positive solution $u\in\Cc^2(\Rd)$
of \cref{EP2} with eigenvalue $\lambda$  and $\mu\in\eom_F$, we have
\begin{equation}\label{ET3.1A}
\mu(F) -\lambda\,=\,
\int_{\Rd\times\Rd} \Bigl(L(x,\xi) -\xi\cdot D{u}(x)
+ \tfrac{1}{\gamma} \babs{ D{u}(x)}^\gamma\Bigr)\,\mu(\D x,\D\xi)\,\ge\,0\,,
\end{equation}
there exists a Borel
probability measure $\nu_u$ on $\fB(\Rd)$, such that, with
$\xi_u$ as defined in \cref{Exiu}, we have
\begin{equation}\label{ET3.1B}
\mu_u(\D{x},\D\xi)\,\df\, \nu_u(\D{x})\delta_{\xi_u(x)}(\D\xi) \,\in\,\eom\,,
\end{equation}
and
\begin{equation}\label{ET3.1C}
\int_{\Rd\times\Rd} \rc(x,\xi)\,\mu_u(\D x,\D\xi) \,=\, \overline\lambda\,.
\end{equation}
In particular, $\lambda=\overline\lambda$.
\item[\textup{(}c\textup{)}]
We have $\eom_F^*=\{\mu_u\}$, with $\mu_u$ as in \cref{ET3.1B}.
\item[\textup{(}d\textup{)}]
There exists at most one positive solution of \cref{EP2}.
\end{enumerate}
\end{theorem}

\begin{proof}
Part (a) follows as in \cref{L2.1} with a slight modification.
Instead of \cref{PL2.1C}, we use the inequality
\begin{equation*}
-\Delta u_n(y) +  \xi_{n,u}(y)\cdot D{u}_n(y) \,\ge \, f_n(y)\,,
\end{equation*}
with $\xi_{n,u}(y) \,\df\, \Gamma_n\, \xi_u\bigl(x_n + \Gamma_n y\bigl)$, and 
$\xi_u$ as in \cref{Exiu}.
Then we apply  \cref{ELF3} and \cref{C3.1}, and follow
the proof of \cref{L2.1}.

For part (b), we define
\begin{equation*}
G(x,\xi,p) \,\df\, H(x,p) -  \xi\cdot p + L(x,\xi)\,\ge\,0\,,
\end{equation*}
and write
\begin{align*}
\Delta u(x) -\xi\cdot  D{u}(x) &\,=\,
\lambda + H(x,D{u}) - f(x) - \xi\cdot  D{u}(x)\\
&\,=\, \bigl(\lambda - L(x,\xi)  - f(x)\bigr) + G(x,\xi,D{u})\,,
\end{align*}
and following the proof of \cref{L2.2}, we obtain
\begin{equation*}
\int_{\Rd\times\Rd} \bigl(L(x,\xi)+f(x)\bigr)\,\mu(\D x,\D\xi) -\lambda\,=\,
\int_{\Rd\times\Rd} G\bigl(x,\xi,D{u}(x)\bigr)\,\mu(\D x,\D\xi)\,\ge\,0\,.
\end{equation*}
The remaining assertions in (b) follow by \cref{L2.3}, with $\xi_u$ as
defined in \cref{Exiu}.

Part (c) follows as in \cref{L2.3} with a slight difference.
For $\gamma>2$ we don't know a priori that $\abs{\xi}^2$ is integrable under
a measure in $\eom_F$.
So instead of the densities $\zeta$ and $\zeta_u$ we use the Radon--Nikodym
derivatives.

Part (d) follows from parts (b) and (c).
This concludes the proof.
\end{proof}

\begin{remark}
\cref{T3.1} does not address existence of a positive solution to \cref{EP2}.
For Hamiltonians not depending on $x$, existence is asserted in \cite{Cirant-14}.
In general, under some additional assumptions,
we can show that there exists a positive solution to \cref{EP2}. 
In addition to \hyperlink{A1}{\sf{(A1)}}--\hyperlink{A2}{\sf{(A2)}},
we also assume that for any bounded $C^2$ domain $D$, there exists a constant $\beta>0$ 
satisfying the following: for every $\varepsilon>0$ there exists $\delta>0$ such that
\begin{equation}\label{ER3.1A}
\abs{H(x, \xi)-\beta |\xi|^\gamma}\leq \varepsilon \left(\abs{\xi}^\gamma
+ \bigl(\dist(x, \partial D)\bigr)^{-\gamma^*}\right), \quad \text{whenever\ }
\dist(x, \partial D)<\delta \text{\ and\ } \xi\in\Rd\,.
\end{equation}
This is same as \cite[(2.23)]{LP16}.
It then follows from \cite[Theorem~2.15]{LP16}
that there exists  a unique 
$v_D\in C^2(D)$ and a constant $c_D$ satisfying
\begin{equation}\label{ER3.1B}
-\Delta v_D(x) + H(x, D v_D(x))  \,=\,
f(x) - c_D\quad \text{in}\; D, \quad \text{and}\quad \lim_{x\to\partial D} v_D
\,=\,\infty\,.
\end{equation}
Furthermore, $c_D$ is characterized as follows:
\begin{equation*}
c_D \,=\, \sup\,\bigl\{c\in\RR\,\colon \exists\; v\in \Sob^{1, 2}(D)\cap \Lp^\infty(D)
\text{\ such that\ } -\Delta v(x) + H(x, D v(x))-f(x) + c\leq 0\bigr\}\,.
\end{equation*}
Thus $c_D$ is monotone decreasing as a function of $D$.
Denote by $(v_n,c_n)$ the solution pair of \cref{ER3.1B} corresponding to $D=B_n$, and
let $x_n\in \Argmin v_n$.
It follows from the equation above that
\begin{equation*}
-c_n -H(x_n, 0) + f(x_n)\,=\,-\Delta v_n(x_n) \,\le\, 0\,,
\end{equation*}
which implies that
\begin{equation*}
c_n \,\ge\, - \min_{\Rd}\, (f-h_0)^-\,.
\end{equation*}
Let $c=\lim_{n\to\infty} c_n$ which exists by the above estimate.
It also clear that $v_n$ attains its minimum in a compact set independent of $n$.
Thus we can follow a standard argument (see \cite[Theorem~2.6]{Barles-16})
to show that $v_n-\min_\Rd v_n\to v$ as $n\to\infty$, and 
\begin{equation*}
-\Delta v(x) + H(x, D v(x))  \,=\, f(x) - c\quad \text{in\ } \Rd\,.
\end{equation*}

Assumption \eqref{ER3.1A} is satisfied by a large class of Hamiltonian.
For instance, consider 
\begin{equation*}
H(x, \xi) \,=\, b(x)\cdot \xi + \frac{1}{\gamma} |\xi|^\gamma
\end{equation*}
for some bounded function $b$.
Then we can choose $\beta= \frac{1}{\gamma}$ above.
Note that \eqref{ER3.1A} follows from the estimate below. 
\begin{align*}
b(x)\cdot \xi &\,\le\, \varepsilon \abs{\xi}^\gamma + C \varepsilon^{-\frac{1}{\gamma-1}}
\\
&\,\le\, \varepsilon \left(\abs{\xi}^\gamma
+ C \varepsilon^{-\frac{\gamma}{\gamma-1}}\right)
\end{align*}
where the constant $C$ depends on $\norm{b}_\infty$ and $\gamma$.
Now choose $\delta = C^{-\frac{\gamma-1}{\gamma}} \varepsilon$.
\end{remark}

\subsection{Remarks on the superquadratic case}\label{S3.1}

Cirant in \cite{Cirant-14} is adopting \hyperlink{A2}{\sf{(A2)}}, except
that in his model the Hamiltonian does not depend on $x$, that is, $H(x,p)=H(p)$.
He also assumes that $f$ and $Df$ have at most polynomial growth.
He shows that there always exists a positive solution to \cref{EP2},
and that this has a at least linear growth.

For this model we can establish that $\int\abs{u}^2\,D\mu<\infty$ for
all $\mu\in\eom_F$, and that therefore, \cref{ET3.1A} holds.
However, the proof of this differs from the proof of \cref{L2.2}.
We choose instead a smooth concave function $\chi$ such that
$\chi(s)= s$ for $s\le0$, and $\chi(s) = 1$ for $s\ge 2$, and we scale it
by defining $\chi^{}_t(s) \df t + \chi(s-t)$ for $t\in\RR$.
Since $\gamma\ge 2$, and $Du$ has polynomial growth, while
$u$ has at least linear growth, we can follow the argument
in the proof of \cite[Theorem~4.1]{ABBK} to conclude that 
$\int\abs{u}^2\,D\mu<\infty$ for all $\mu\in\eom_F$.
In \cite{Cirant-14}, the set of admissible controls are required to
satisfy $\limsup_{T\to\infty}\,\frac{\Exp^x[u(X_t)]}{T}=0$.
This is an unnecessary restriction on the class of admissible controls,
and can be avoided.
Without assuming that $H(p)$ is strictly convex, which might result in
non-uniqueness for $u$,
the approach summarized above,
shows that $\xi_u$ is an optimal Markov control and the corresponding infinitesimal
measure is a minimizer of \cref{LP}.
Thus, we have a strong notion of optimality as explained in \cref{S4}.
Under the additional assumption that $H(p)$ is strictly convex, the positive
solution $u$, and therefore also the optimal Markov control are unique.

%
%
%

\section{Implications for the ergodic control problem}\label{S4}

The problem \cref{EP} is associated with an ergodic control problem
for the diffusion $X=(X_t)_{t\ge0}$ given by the It\^o stochastic
differential equation
\begin{equation}\label{SDE}
\D{X_t} \;=\;
-\xi_t\,\D{t} + \sqrt{2}\,\D{W_t} \,, \quad X_0 = x\in \RR^d\,.
\end{equation}
This equation is specified on a complete, filtered probability space
$\bigl(\Omega,\sF,\Prob, (\sF_t)_{t\ge0}\bigr)$, with $(W_t)_{t\ge0}$
an $(\sF_t)$-adapted $d$-dimensional Brownian motion.
An \emph{admissible control} is an $\Rd$-valued $(\sF_t)$-progressively measurable
process $\xi=(\xi_t)_{t\ge0}$, such that
$\Exp\Bigl[\int_0^T \abs{\xi_t}^{\gamma^*}\,\D{t}\Bigr] < \infty$ for all $T>0$,
and we let $\Uadm$ denote the class of such controls.
The \emph{running cost} function is given by $F$ in \cref{EF}.

In \cite{Ichihara-12}, optimality is established via the
study of the parabolic problem.
In view of the optimality results concerning \cref{LP}, we
can state a stronger version of optimality.
We state this result for the model in \cref{EP}, noting that
an identical argument can be used to establish this for
\cref{EP2} under \hyperlink{A2}{\sf{(A2)}}.
We let $\Exp^x_\xi$ denote the expectation operator for the diffusion
in \cref{SDE} controlled by $\xi\in\Uadm$ with initial condition $X_0=x$.

\begin{theorem}\label{T4.1}
Assume \hyperlink{A1}{{\sf(A1)}}.
Let $u\in\Cc^2(\Rd)$ be the unique positive solution of \cref{EP}
in the subquadratic
case, or as in \cref{S3.1} for the superquadratic case.
With $\xi_u$ as in \cref{EL2.3A}, we have
\begin{equation}\label{ET4.1A}
\liminf_{T\to\infty}\,\inf_{\xi\in\Uadm}\,
\Exp^x_{\xi}\biggl[\int_0^T F(X_t,\xi_t) \,\D{t}\biggr]\,\ge\, \overline\lambda
\,=\,
\limsup_{T\to\infty}\,
\Exp^x_{\xi_u}\biggl[\int_0^T F\bigl(X_t,\xi_u(X_t)\bigr) \,\D{t}\biggr]\,.
\end{equation}
Moreover, \cref{ET4.1A} holds without the expectation operators in the a.s.\ 
pathwise sense.
In addition a Markov control $v\colon\Rd\to\Rd$ is optimal, if and only
if it agrees with $\xi_u$ a.e.\ in $\Rd$.
\end{theorem}

\begin{proof}
The inequality in \cref{ET4.1A} follows from the fact that
limit points of mean empirical measures in $\cP(\Rd\times\Rd)$ are
infinitesimal measures for the operator $\Ag$ (see Lemma~3.4.6 and Theorem~3.4.7
in \cite{ABG12}) together with the definition of $\overline\lambda$ in
\cref{LP}.
The equality follows by the ergodicity of the process under the control
$\xi_u$ and the fact that $F$ is integrable under the
invariant probability measure as asserted in \cref{L2.3}.
The pathwise results also follow from results in \cite{ABG12} referenced above.
The verification part of the theorem follows from \cref{T2.1}.
\end{proof}

\section*{Acknowledgements}
The work of Ari Arapostathis and Luis Caffarelli was supported in part by
the National Science Foundation through grant DMS-1715210.
The work of Ari Arapostathis was also supported in part by
the Army Research Office through grant W911NF-17-1-001.
The research of Anup Biswas was supported in part by an INSPIRE faculty fellowship
and DST-SERB grant EMR/2016/004810.



\end{document}